\documentclass{amsart}
\usepackage{graphicx} 
\numberwithin{equation}{section}
\usepackage{amsmath}
\usepackage{amssymb}
\usepackage{amsthm}
\usepackage{color}
\usepackage{enumerate}
\theoremstyle{definition}
\newtheorem{thm}{Theorem}[section]
\newtheorem{prop}[thm]{Proposition}

\newtheorem{lemma}[thm]{Lemma}
\newtheorem{definition}[thm]{Definition}
\title[random iterated function systems]{Bowen's formula and the difference between random iterated function systems and random recursive constructions}
\author{Yuya Arima}
\date{\today}
\address{Graduate School of Mathematics, Nagoya University,
Furocho, Chikusaku, Nagoya, 464-8602, JAPAN} 
\email{yuya.arima.c0@math.nagoya-u.ac.jp}
\subjclass[2020]{28A80, 37H12}
\thanks{{\it Keywords}: Random subset, Cantor set, Hausdorff dimension}

\newcommand{\fram}{\mathcal{F}}

\newcommand{\edgeset}{I}
\newcommand{\seq}{r}

\begin{document}

\begin{abstract}
In this paper, we show that the difference in independence structures between random iterated function systems and random recursive constructions is reflected in the validity of Bowen's formula. More precisely, we construct random iterated function systems for which Bowen's formula fails, whereas Bowen's formula always holds for random recursive constructions. Our construction demonstrates that the difference in independence structures between the two random models has genuine consequences for dimension theory.
\end{abstract}

\maketitle

\section{Introduction}
Let $I$ be a countable index set and let $X$ be a convex compact subset of $\mathbb{R}^d$ ($d\in\mathbb{N}$) such that $X$ is the closure of its interior in $\mathbb{R}^d$. Let
$\Psi:=\{\psi_i\}_{i\in I}$ be an iterated function system consisting of contracting self-similarities $\psi_i$ on $X$. 
Assume that $\Psi$ satisfies the open set condition, that is, for each $i,j\in I$ with $i\neq j$ we have $\psi_i(\text{Int}(X))\cap\psi_j(\text{Int}(X))=\emptyset$, where $\text{Int}(X)$ denotes the Euclidean interior of $X$. 
The limit set 
\[
J(\Psi):=\bigcap_{n\in\mathbb{N}}\bigcup_{\tau\in I^n}\psi_{\tau_1}\circ\cdots\circ\psi_{\tau_n}(X).
\]
of $\Psi$ is, in general, a fractal set and is one of the most studied classes of fractal sets. Bowen’s formula yields that the Hausdorff dimension with respect to the Euclidean metric on $\mathbb{R}^d$ (see \cite{Falconerbook} for the definition) of $J(\Psi)$ is given by 
\[
\inf \left\{t>0:\sum_{i\in I}c_i^t\leq 1\right\}=\inf \left\{t>0:\log\sum_{i\in I}c_i^t\leq 0\right\},
\]
where $c_i$ denotes the contraction ratio of $\psi_i$.
This is a fundamental property of the Hausdorff dimension of $J(\Psi)$ and allows us to study the Hausdorff dimension of $J(\Psi)$ by considering only the natural covers generated by $\Psi$, namely, 
\[
\{\psi_{\tau_1}\circ\cdots\circ\psi_{\tau_n}(X)\}_{\tau\in I^n}, \ n\in\mathbb{N}.
\]
It is well-known (see, for example, \cite{Falconerbook, mauldin2003graph}) that if $\Psi$ satisfies the open set condition then Bowen’s formula always holds. 
In this paper, we show that the difference in independence structures between random iterated function systems and random recursive constructions is reflected in the validity of the random analogue of Bowen's formula.

Random fractal subsets of the $d$-dimensional Euclidean space $\mathbb{R}^d$ ($d\in\mathbb{N}$) have attracted significant attention as models that are closer to natural phenomena than fractal sets generated by deterministic iterated function systems.
There are two well-known random constructions. The first is known as random iterated function systems (RIFSs), and the second is referred to as random recursive constructions. 
In particular, the dimensional properties of random fractal sets constructed by these methods have been extensively studied. See \cite{Yuya, Falconerbook, Falconerrandomfractals, Graf, Hambly, kifer2006random, Kiferlong, LiuWu, Mauldin, mayer2011distance, RGDMS, Trocheit}.

The independence in the choice of IFSs in random recursive constructions can be regarded as stronger than that in RIFSs (see Section \ref{sec statement} for details).
However, to the best of our knowledge, despite the fundamental difference in the independence structures of the two models, dimensional results established for random recursive constructions have also been established for RIFSs.
In contrast, in this paper we show that such a correspondence does not hold in general by providing an example of a random iterated function system for which the random analogue of Bowen's formula fails to hold. We believe that this is the first result to show that the difference in independence structures between random iterated function systems and RIFSs has concrete consequences for dimension theory.


\subsection{Statement of the main theorem}\label{sec statement}
Let $d\in\mathbb{N}$ and let $X$ be a convex compact subset of $\mathbb{R}^d$ such that $X$ is the closure of its interior in $\mathbb{R}^d$. 
Let $\Psi^{(i)}$ ($i\in\mathbb{N}$) be a set of contracting  self-similarities $\{\psi_{j}^{(i)}:X\rightarrow X\}_{j\in I^{(i)}}$, where $I^{(i)}$ is a countable index set with $\#I^{(i)}\geq 2$, such that for all $i\in\mathbb{N}$ and $j,\tilde j\in I^{(i)}$ with $j\neq \tilde j$ we have
\[
\psi_{j}^{(i)}(\text{Int($X$)})\cap\psi_{\tilde j}^{(i)}(\text{Int($X$)})=\emptyset.
\]
We call $\Psi^{(i)}$ an iterated function system (IFS).
For $i\in\mathbb{N}$ and $j\in I^{(i)}$ let $0<c_{j}^{(i)}<1$ be the contraction ratio of $\psi_{j}^{(i)}$, that is, for $x,y\in X$ we have 
\[
{|\psi_{j}^{(i)}(x)-\psi_{j}^{(i)}(y)|}=c_{j}^{(i)}{|x-y|}.
\]
    We consider a family $\Psi:=\{\Psi^{(i)}\}_{i\in\mathbb{N}}$ of iterated function systems. We assume that there exists $0<\eta<1$ such that for all $i\in\mathbb{N}$ and $j\in I^{(i)}$ we have 
\[
c_{j}^{(i)}< \eta.
\]
We take a probability vector 
\[
\vec{p}:=(p_1,p_2,\cdots),
\]
that is, $\sum_{i\in\mathbb{N}}p_i=1$ and $p_i\in[0,1]$ for all $i\in\mathbb{N}$.

We first explain RIFSs. 
Let $\Omega:=\mathbb{N}^{\mathbb{N}}$.
We set $\mathbb{N}^*:=\bigcup_{n=1}^\infty\mathbb{N}^n$. For $n\in\mathbb{N}$ and $\omega\in \mathbb{N}^n$ we define $|\omega|:=n$.
We endow $\Omega$ with the $\sigma$-algebra $\mathcal{B}$ generated by the cylinders $\{[\omega]\}_{\omega\in \mathbb{N}^*}$, where $[\omega]:=\{\tilde \omega\in \Omega:\omega_i=\tilde \omega_i, 1\leq i\leq |\omega|\}$.
We consider the Bernoulli measure $\mathbb{P}:=\mathbb{P}_{\vec{p}}$ on the probability space $(\Omega, \mathcal{B})$ satisfying, for each $\omega\in \mathbb{N}^*$ we have 
\begin{align*}
    \mathbb{P}([\omega])={p_{\omega_1}}p_{\omega_2}\cdots p_{\omega_{|\omega|}}.
\end{align*}
The pair $(\vec{p},\Psi)$ is called a random iterated function system (RIFS). 
The RIFS $(\vec{p},\Psi)$ is said to be finite if for all $i\in\mathbb{N}_{\vec{p}_+}:=\{i\in\mathbb{N}:p_i>0\}$ we have $\#I^{(i)}<\infty$.
The random limit set generated by $(\vec{p},\Psi)$ is constructed by choosing the IFS $\Psi^{(i_k)}$ ($k\in\mathbb{N}$) that is applied at the $k$-th level according to the probability vector $\vec{p}$.
Note that this choice of IFS is uniform for that $k$-th level. 
The limit set along $\omega=(\omega_1,\omega_2,\cdots)\in \Omega$ can be written as
\[
 J(\Psi(\omega)) =\bigcap_{n=1}^{\infty}\bigcup_{\tau\in \Sigma_\omega^n}\psi_{\tau}^{(\omega)}(X),
\text{ where } 
    \Sigma_\omega^n:=\prod_{i=1}^{n} I^{(\omega_i)}
\text{ and }
\psi_{\tau}^{(\omega)}:=\psi_{\tau_1}^{(\omega_1)}\circ\cdots\circ\psi_{\tau_n}^{(\omega_n)}.
\]
We define the Bowen parameter by
\begin{align*}
&B(\Psi):=\inf
\left\{
 t \geq 0: 
E_{i\in\mathbb{N}}\left(\log \sum_{j\in I^{(i)}}\left(c_{j}^{(i)}\right)^ t 
\right):=\sum_{i\in\mathbb{N}}p_i \log \sum_{j\in I^{(i)}}\left(c_{j}^{(i)}\right)^ t  \leq 0
\right\}.
\end{align*}

By Roy and Urba\'{n}ski \cite[Theorem 3.26]{RGDMS}, Rempe-Gillen and Urba\'{n}ski \cite[Theorem 9.1]{nonautonomous} and \cite[Theorem 3.8]{Yuya}, we have the following result. Assume that $\Psi$ satisfies the following: 
    For all $i\in\mathbb{N}$ we have $I^{(1)}=I^{(i)}$ and if $\#I^{(1)}=\infty$ then for all $j\in I^{(1)}$ we have 
    $\inf_{i\in \mathbb{N}_{\vec{p}_+}}{c_{j}^{( i)}}>0$.
Then, for $\mathbb{P}$-a.s. $\omega\in\Omega$ we have
\[
\dim_H( J(\Psi(\omega)) )=B(\Psi),
\]
where $\dim_H( J(\Psi(\omega)) )$ denotes the Hausdorff dimension of $ J(\Psi(\omega)) $ with respect to the Euclidean metric on $\mathbb{R}^d$ (see also Kifer \cite{Kiferlong} and Hambly \cite{Hambly}).

Next, we briefly explain random recursive constructions. For detailed mathematical descriptions, we refer the reader, for example, to Mauldin and Williams \cite{Mauldin} and to Falconer's textbook \cite[Section 15]{Falconerbook}.
In random recursive constructions, the limit set is constructed in a recursive manner by assigning the IFS $\Psi^{(i_\tau)}$ ($i_\tau\in\mathbb{N}$) chosen according to the probability vector $\vec p$ to every finite word $\tau$ that has already been constructed. 
Note that, while in RIFSs the choice of an IFS is independent only across levels and is made uniformly for all words of the same length, in random recursive constructions this choice of the IFS is independent for all distinct words.
This implies that random recursive constructions exhibit a stronger form of independence in the choice of IFSs than RIFSs.
Mauldin and Williams \cite[Theorem 1.1]{Mauldin} showed that 
    \begin{align}\label{thm mouldin}
    &\text{the Hausdorff dimension of the limit set constructed in}
    \\&
    \text{ such a way is a.s. given by }
\inf \left\{ t \geq 0: \sum_{i\in\mathbb{N}} p_i\sum_{j\in I^{(i)}} \left(c_{j}^{(i)}\right)^ t  \leq 1
\right\}.\nonumber
\end{align}
Note that we obtained the above result without making any assumptions on $\Psi$.
However, the following main theorem shows that, for RIFSs Bowen’s formula does not hold in general.
\begin{thm}\label{thm main}
    There exists a finite random iterated function system $(\vec{p},\Psi)$ such that for $\mathbb{P}$-a.s. $\omega\in \Omega$ we have
    \[
    \dim_H( J(\Psi(\omega)) )<B(\Psi).
    \]
\end{thm}
As mentioned above, while in RIFSs the choice of an IFS is independent only across levels and is made uniformly for all words of the same length, in random recursive constructions the choice is independent for all distinct words.
This reflects a fundamental difference in the independence structures of random recursive constructions and RIFSs.  
However, analogous results on fractal dimensions established for random recursive constructions have also been obtained for RIFSs.
See, for example, \cite{Falconerbook, Falconerrandomfractals, Graf, LiuWu, Trocheit, Mauldin} for results on Random recursive constructions, and 
\cite{Hambly, Kiferlong, nonautonomous, RGDMS, Trocheit, Yuya}
 for results on RIFSs.
In contrast, Theorem \ref{thm main} together with \eqref{thm mouldin} shows the following: Let $\vec{p}$ be the probability vector and let $\Psi$ be the family of IFSs constructed in Theorem \ref{thm main}. 
Although Bowen's formula holds for $\vec{p}$  and $\Psi$ in random recursive constructions by \eqref{thm mouldin}, it fails for $\vec{p}$ and $\Psi$ in RIFSs by Theorem \ref{thm main}.  
This implies that analogous results established for random recursive constructions are not always obtained for RIFSs. 
This demonstrates that the difference in independence structures between random recursive constructions and RIFSs has genuine consequences for dimension theory.

In non-autonomous settings, Rempe-Gillen and Urbański \cite{nonautonomous} constructed examples for which a version of Bowen's formula fails.
However, random settings differ fundamentally from non-autonomous settings. 
In fact, in non-autonomous settings, one can freely prescribe the sequence of IFSs to construct examples for which Bowen's formula fails. By contrast, in random settings, the sequence of IFSs is determined by the underlying random process and therefore cannot be prescribed arbitrarily.
Hence, although examples for which a version of Bowen's formula fails are known in non-autonomous settings, this does not make it straightforward to construct random iterated function systems for which the random analogue of Bowen's formula fails almost surely.
Indeed, for the probability vector $\vec{p}$ and the family of IFSs $\Psi$  constructed in Theorem \ref{thm main}, the random analogue of Bowen's formula holds almost surely in the setting of random recursive constructions by \eqref{thm mouldin}. 
Consequently, the failure of Bowen's formula in our example cannot be explained merely by the construction used in non-autonomous settings. Rather, it reflects the fundamentally different independence structures of the two random models.

\section{Proof of the main theorem}
Let $d\geq 1$ and let $X:=[0,1]^{d}$. 
We denote by
$(\boldsymbol{e}_1,\boldsymbol{e}_2,\cdots,\boldsymbol{e}_{d})$  the canonical base of $\mathbb{R}^d$.
For each $\boldsymbol{i}=(i_1,i_2,\cdots,i_d)\in \{0,1\}^d$ we 
define the map $\phi_{\boldsymbol{i}}:X\rightarrow X$ by 
\[
\phi_{\boldsymbol{i}}(x)=\frac{1}{2}x+\frac{1}{2}v_{\boldsymbol{i}}, \text{ where }v_{\boldsymbol{i}}:=\sum_{\ell=1}^{d}i_\ell\boldsymbol{e}_\ell.
\]
We define the index sets $I_1$ and $I_{2^d}$ by 
\[
I_1:=\{0\}^{d}
\text{ and }
I_{2^d}:=\{0,1\}^d
\]
\begin{definition}
    A  pair $\fram=(\{U_n\}_{n\in\mathbb{N}},\{V_n\}_{n\in\mathbb{N}})$ of sequences of positive integers is called a frame if  $\fram$ satisfies the following conditions:
    \begin{itemize}
        \item[(F)] For all $n\in\mathbb{N}$ we have 
        $nU_n\leq V_n 
            \text{ and }
            (U_n+V_n)^3\leq U_{n+1}.$
    \end{itemize}
\end{definition}
We consider a fixed frame $\mathcal{F}$ throughout this section. 
For each $i\in\mathbb{N}$ we define 
\begin{align*}
    I^{(i)}:=\edgeset(\fram)^{(i)}:=I^{U_i}_{1}\times I_{2^d}^{V_i}.
\end{align*}
For each $i\in\mathbb{N}$ and $\tau=(\tau_1,\cdots,\tau_{U_i+V_i})\in \edgeset^{(i)}$ we define
\begin{align}\label{eq def psi}
\psi_\tau^{(i)}:=\phi_{\tau_1}\circ\cdots\circ\phi_{\tau_{U_n+V_n}}
\text{ and }
\Psi^{(i)}:=\Psi(\fram)^{(i)}:=\{\psi_{\tau}^{(i)}\}_{\tau\in \edgeset^{(i)}}.    
\end{align}
We take the probability vector $\vec{p}:=(p_1,p_2,\cdots)$ such that for all $n\in\mathbb{N}$ we have
\begin{align}\label{eq probability vector}
p_n=\frac{1}{Cn^2},\text{ where }C:=\sum_{n=1}^\infty\frac{1}{n^2}.    
\end{align}
Let $(\Omega,\mathcal{B},\mathbb{P})$ be the probability space as defined in the introduction. Throughout the remainder of this paper, we use the notation introduced in the introduction.
\begin{prop}\label{prop pressure}
Let $t\in [0,\infty)$. For $\mathbb{P}$-a.s. $\omega\in \Omega$ we have
\begin{align}\label{eq pressure}
E_{i\in\mathbb{N}}\left(\log \sum_{j\in I^{(i)}}
\left(c_{j}^{(i)}\right)^t
\right)
=\left\{
 \begin{array}{cc}
   \infty   & \text{if}\ t<d\\
   -\infty   & \text{if}\ t\geq d 
 \end{array}
 \right.   .      
\end{align}
In particular, we have
$B(\Psi)=d$
\end{prop}
\begin{proof}
Let $t\in [0,\infty)$. Notice that we have 
\[
\sum_{i=1}^\infty p_i \log \sum_{j\in I^{(i)}}\left(c_j^{(i)}\right)^t=\frac{\log2}{C}\sum_{i=1}^\infty\frac{-tU_i+(d-t)V_i}{i^2}
\]
By definition of the frame, for all $t\geq d$ we obtain
\begin{align*}
\sum_{i=1}^\infty\frac{-tU_i+(d-t)V_i}{i^2}
\leq
\sum_{i=1}^{\infty}\frac{-dU_{i}}{i^2}
    \leq
    \sum_{i=1}^{\infty}\frac{-d}{i}=-\infty.
\end{align*}
Next, we consider the case $0\leq t<d$. Let $0\leq t<d$.
We take a large number $M_t\geq 1$ such that for all $i\geq M_t$ we have $-t+(d-t)i\geq 1$. By the definition of the frame,  we have
\begin{align*}
\sum_{i=1}^{\infty}\frac{-tU_i+(d-t)V_i}{i^2}
\geq\sum_{i=1}^{\infty}\frac{(-t+(d-t)i)U_i}{i^2}
\geq
D_t
+
\sum_{i=M_t}^{\infty}\frac{1}{i}=\infty
,
\end{align*}
where $D_t:=\sum_{i=1}^{M_t-1}({(-t+(d-t)i)U_i})/{i^2}$.
\end{proof}

Next, we shall show that for $\mathbb{P}$-a.s. $\omega\in \Omega$ we have
    $\dim_H(J(\Psi(\omega)))=0$.
The proof of this is divided into several lemmas.

For each $i\in\mathbb{N}$ we define the random variable $X_i:\Omega\rightarrow \mathbb{N}$ by
\[
X_i(\omega)=\omega_i.
\]
Note that the family of events $\{\{X_i\geq i\}\}_{i\in\mathbb{N}}$ is independent and 
\begin{align*}
    \sum_{i=1}^\infty\mathbb{P}(\{X_i\geq i\})=\sum_{i=1}^{\infty}\sum_{k=i}^\infty\frac{1}{Ck^2}=\sum_{k=1}^\infty\frac{1}{Ck}=\infty.
\end{align*}
Therefore, by Borel-Cantelli's Lemma, we obtain 
\[
\mathbb{P}(\Omega_\infty)=1, \text{ where }\Omega_\infty:=\bigcap_{i\in\mathbb{N}}\bigcup_{k\geq i}\{X_k\geq k\}.
\]
In particular, for all $\omega\in \Omega_\infty$ there exists $\{n_l\}_{l\in\mathbb{N}}\subset\mathbb{N}$ satisfying  $n_l<n_{l+1}$ and 
$    X_{n_l}(\omega)\geq n_l$ for all $l\in\mathbb{N}$.

\begin{lemma}\label{lemma tequnical}
Let $\omega\in \Omega_\infty$. Then,
there exist sequences $\{\seq_n\}_{n\in\mathbb{N}}\subset \mathbb{N}$ and $\{b_n\}_{n\in\mathbb{N}} \subset \mathbb{N}$ such that we have the following:
\begin{itemize}
    \item[(S1)] For all $n\in\mathbb{N}$ we have $b_n\leq \seq_n$.
    \item[(S2)] For all $n\in\mathbb{N}$ we have  $X_{b_n}(\omega)\geq \seq_n$
    \item[(S3)] For all $n\in\mathbb{N}$ we have $\max_{1\leq k\leq b_n-1}X_k(\omega)<X_{b_n}(\omega)$ if $b_n>1$ and $X_1(\omega)=X_{b_n}(\omega)$ otherwise.
    \item[(S4)] For all $n\in\mathbb{N}$ we have $\seq_{n}<\seq_{n+1}$ and  $b_{n}<b_{n+1}$.
\end{itemize}
\end{lemma}

\begin{proof}
Fix $\omega\in \Omega_\infty$. Then, there exists $\{n_l\}_{l\in\mathbb{N}}\subset\mathbb{N}$ such that for all $l\in\mathbb{N}$ we have $n_l<n_{l+1}$ and 
\begin{align}\label{eq average}
    X_{n_l}(\omega)\geq n_l.
\end{align}
    We will construct sequences $\{\seq_n\}_{n\in\mathbb{N}}$ and $\{b_n\}_{n\in\mathbb{N}}$ satisfying desired conditions inductively. 
    Let $\seq_1:=n_1$ and let 
    \[
    b_1:=\min\left\{i\in\mathbb{N}:i\leq \seq_1,\ X_i(\omega)=\max _{1\leq k\leq \seq_1}X_k(\omega)\right\}.
    \]
    Then, by (\ref{eq average}), $\seq_1$ and $b_1$ satisfy (S1), (S2) and (S3) for $n=1$. 

    Since for all $l\in\mathbb{N}$ we have $n_l<n_{l+1}$, there exists $l_2\in\mathbb{N}$ such that $n_{l_2}>\seq_{1}$ and  $n_{l_2}\geq X_{b_1}(\omega)+1$.
    We set $\seq_2:=n_{l_2}$
and
\[
b_2:=\min\left\{i\in\mathbb{N}:i\leq \seq_2,\ X_i(\omega)=\max _{1\leq k\leq \seq_2}X_k(\omega)\right\}.
\]
Then, we have $b_2\leq \seq_2$, $\max_{1\leq k<b_2-1}X_k(\omega)<X_{b_2}(\omega)$ and $\seq_1<\seq_2$. 
By (\ref{eq average}), we have $X_{b_2}(\omega)\geq \seq_2$.
Therefore, 
since $\max_{1\leq k\leq b_1} X_k(\omega)=X_{b_1}(\omega)<\seq_2$, we have $b_1<b_2$.
 Hence, $\{\seq_1,\seq_2\}$ and $\{b_1,b_2\}$ satisfy desired conditions for $1\leq n\leq 2$. 

Let $j\geq 2$.
We assume that sequences $\{\seq_n\}_{n=1}^{j}$ and $\{b_n\}_{n=1}^{j}$ satisfying desired conditions for all $1\leq n\leq j$ are already defined.
Then,  there exists $l_{j+1}\in\mathbb{N}$ such that 
$n_{l_{j+1}}>\seq_{j}$ and
$n_{l_{j+1}}\geq X_{b_j}(\omega)+1$.
We set $\seq_{j+1}:=n_{l_{j+1}}$ and 
\[
b_{j+1}:=\min\left\{i\in\mathbb{N}:i\leq \seq_{j+1},\ X_i(\omega)=\max _{1\leq k\leq \seq_{j+1}}X_k(\omega)\right\}.
\]
As in the argument above, we can show that $\{\seq_n\}_{n=1}^{j+1}$ and $\{b_n\}_{n=1}^{j+1}$ satisfy the desired conditions for all $1\leq n\leq j+1$. Thus, we are done.
\end{proof} 

Let $\omega\in\Omega_\infty$. For $i\in\mathbb{N}$ and $1\leq k\leq U_{\omega_i}+V_{\omega_i}$ we set $\edgeset^{(\omega_i,k)}=I_1$ if $1\leq k\leq U_{\omega_i}$ and $I^{(\omega_i,k)}=I_{2^d}$ if $U_{\omega_i}+1\leq k\leq V_{\omega_i}+U_{\omega_i}$.
Then, for all $i\in\mathbb{N}$ we have
\begin{align}\label{eq index relation}
\edgeset^{(\omega_i)}=\prod_{\ell=1}^{U_{\omega_i}+V_{\omega_i}} \edgeset^{(\omega_i,\ell)}.
\end{align}
We consider the non-autonomous conformal iterated function system 
\begin{align*}
&\Phi_\omega:=\{
\Phi^{(\omega_1,1)},\cdots,\Phi^{(\omega_1,U_{\omega_1}+V_{\omega_1})},
\cdots,
\Phi^{(\omega_i,1)},\cdots,\Phi^{(\omega_i,U_{\omega_i}+V_{\omega_i})},\cdots\},\text{ where }
\\&
\Phi^{(\omega_i,k)}:=\{\phi_{\boldsymbol{i}}\}_{\boldsymbol{i}\in \edgeset^{(\omega_i,k)}}
\text{ for each }
i\in \mathbb{N} \text{ and }1\leq k\leq U_{\omega_i}+V_{\omega_i}.
\end{align*}
For  $1\leq n\leq U_{\omega_1}+V_{\omega_1}$ we set $\widetilde\Sigma_\omega^n=\prod_{\ell=1}^n\edgeset^{(\omega_1,\ell)}$.
Also, for $n=\sum_{i=1}^{m-1}(U_{\omega_i}+V_{\omega_i})+k$ with $m\geq 2$ and $1\leq k\leq U_{\omega_m}+V_{\omega_m}$ we set 
\[
\widetilde\Sigma_\omega^n:=
   \prod_{i=1}^{m-1}\left(\prod_{\ell=1}^{U_{\omega_i}+V_{\omega_i}}\edgeset^{(\omega_{i},\ell)}\right)  \times \prod_{\ell=1}^kI^{(\omega_m,\ell)}.
   \]
By (\ref{eq index relation}), for all $m\in\mathbb{N}$ and $j_m=\sum_{i=1}^m(U_{\omega_i}+V_{\omega_i})$ we have 
$\Sigma_{\omega}^m=\tilde \Sigma_\omega^{j_m}$.
For $n\in\mathbb{N}$ and $\widetilde\tau \in \widetilde\Sigma_\omega^n$ we set 
$
\phi_{\widetilde\tau}^n:=\phi_{\widetilde\tau_1}\circ\cdots\circ\phi_{\widetilde\tau_n}
$ and $c_{\tilde \tau}=2^{-n}$.
By (\ref{eq def psi}) and (\ref{eq index relation}), we have 
\begin{align}\label{eq coincide limit set}
    J(\Phi_\omega):=\bigcap_{n=1}^{\infty}\bigcup_{\widetilde \tau\in \widetilde \Sigma_\omega^n}\phi^n_{\widetilde\tau}(X)=J(\Psi(\omega)).
\end{align}

 \begin{prop}\label{prop Hausdorff zero}
    For $\mathbb{P}$-a.s. $\omega\in \Omega$ we have
$    \dim_H(J(\Psi(\omega)))=0.$
\end{prop}
\begin{proof}
By (\ref{eq coincide limit set}), it is enough to show that for all $\omega\in \Omega_\infty$ we have $\dim_H(J(\Phi_\omega))=0$.
Let $\omega\in \Omega_\infty$.
By \cite[Lemma 2.8]{nonautonomous}, we have
\begin{align}\label{eq weak bowen}
    \dim_H(J(\Phi_\omega))\leq \inf \left\{t\geq 0:P(t):=\liminf_{n\to\infty}\frac{1}{n}\log\sum_{\widetilde\tau\in\widetilde\Sigma_{\omega}^n}
    c_{\tilde \tau}^t<0
    \right\}.
\end{align}
We will show that for all $t\geq 0$ we have 
$
P(t)\leq -t\log 2<0.
$
For all $n\in\mathbb{N}$ we set
$
j_n:=\sum_{i=1}^{n}(U_{\omega_i}+V_{\omega_i}).
$
Let $n\geq 2$ and let $t\geq 0$.
We have
\begin{align}\label{eq proof of prop 1}
    &\frac{1}{j_{b_n-1}+U_{\omega_{b_n}}}
    \log
    \sum_{\widetilde\tau\in
    \widetilde\Sigma_{\omega}^{j_{b_n-1}+U_{\omega_{b_n}}}}
    c_{\tilde\tau}^t
    \leq -t\log2+
    \frac{j_{b_n-1}\log 2}{j_{b_n-1}+U_{\omega_{b_n}}}
    \end{align}
By (S1) and (S2) of Lemma \ref{lemma tequnical}, we have 
$
b_n\leq \seq_n\leq \omega_{b_n}.
$
By (S3) of Lemma \ref{lemma tequnical}, we have $\max\{ \omega_i:{1\leq i\leq b_n-1}\}<\omega_{b_n}$. This implies that 
\[
j_{b_n-1}\leq b_n(U_{\omega_{b_n}-1}+V_{\omega_{b_n}-1})
\leq 2b_nV_{\omega_{b_n}-1}
\leq 2\omega_{b_n}V_{\omega_{b_n}-1}.
\]
By the definition of the frame, we have $k+1\leq V_k$ for all $k\geq 2$. Hence, by the definition of the frame, we obtain
\begin{align*}
    \frac{U_{\omega_{b_n}}}{j_{b_n-1}}
    \geq \frac{V_{\omega_{b_n}-1}^3}{2V_{\omega_{b_n}-1}^2}\geq\frac{V_{\omega_{b_n}-1}}{2} 
\text{ and thus, } \lim_{n\to\infty}\frac{U_{\omega_{b_n}}}{j_{b_n-1}}=\infty
\end{align*}
Therefore, by (\ref{eq proof of prop 1}), we obtain
$P(t)\leq -t\log 2<0$.
Hence, by (\ref{eq weak bowen}), for all $\omega\in \Omega_\infty$ we have
$
\dim_H(J(\Psi(\omega)))=0.
$
\end{proof}
Combining Proposition \ref{prop pressure} and Proposition \ref{prop Hausdorff zero}, we obtain the following theorem:
\begin{thm}
    Let $\mathcal{F}$ be a frame and let $\vec{p}$ be the probability vector such that $p_n=(Cn^2)^{-1}$ for all $n\in\mathbb{N}$. 
    Let $\Psi:=\Psi(\mathcal{F}):=\{\Psi(\mathcal{F})^{(i)}\}_{i\in\mathbb{N}}$.
    Then, for $\mathbb{P}_{\vec{p}}$-a.s. $\omega\in\Omega$ we have $\dim_H(J(\Psi(\omega)))<B(\Psi)$.
\end{thm}

 \subsection*{Acknowledgments}
This work was supported by the JSPS KAKENHI 25KJ1382.

 \bibliographystyle{abbrv}
\bibliography{reference}
 \nocite{*}

\end{document}